\documentclass[twoside,a4paper,12pt,centertags]{amsart}
\usepackage{amsmath,amssymb,verbatim,vmargin}
\usepackage[bookmarks=true]{hyperref}   

\theoremstyle{plain}
\newtheorem{thm}{Theorem}[section]
\newtheorem{lem}[thm]{Lemma}

\newtheorem{prop}[thm]{Proposition}

\theoremstyle{definition}
\newtheorem{defn}[thm]{Definition}

\mathchardef\semic="303B

\newcommand{\wedg}{\mathbin{\scriptstyle{\wedge}}}
\newcommand{\lctr}{\mathbin{\lrcorner}}
\newcommand{\R}{{\mathbf R}}
\newcommand{\C}{{\mathbf C}}
\newcommand{\Z}{{\mathbf Z}}
\newcommand{\mH}{{\mathcal H}}
\newcommand{\mD}{{\mathcal D}}

\newcommand{\mL}{{\mathcal L}}

\DeclareMathOperator{\re}{Re}

\newcommand{\sett}[2]{ \{ #1 \, \semic \, #2 \} }

\newcommand{\supp}{\text{{\rm supp}}\,}
\newcommand{\dist}{\text{{\rm dist}}\,}

\newcommand{\nul}{\textsf{N}}
\newcommand{\ran}{\textsf{R}}
\newcommand{\dom}{\textsf{D}}

\newcommand{\clos}[1]{\overline{#1}}

\newcommand{\barint}{\mbox{$ave \int$}}
\newcommand{\divv}{{\text{{\rm div}}}}
\newcommand{\curl}{{\text{{\rm curl}}}}

\newcommand{\ta}{{\scriptscriptstyle \parallel}}
\newcommand{\no}{{\scriptscriptstyle\perp}}
\newcommand{\pd}{\partial}

\newcommand{\linj}[1]{\underline{#1}}
\newcommand{\lf}{\overline{f}}
\newcommand{\hf}{\hat f}

\newcommand{\loc}{\text{{\rm loc}}}
\newcommand{\tN}{\widetilde N_*}

\def\barint_#1{\mathchoice
            {\mathop{\vrule width 6pt
height 3 pt depth -2.5pt
                    \kern -8.8pt
\intop}\nolimits_{#1}}%
            {\mathop{\vrule width 5pt height
3 pt depth -2.6pt
                    \kern -6.5pt
\intop}\nolimits_{#1}}%
            {\mathop{\vrule width 5pt height
3 pt depth -2.6pt
                    \kern -6pt
\intop}\nolimits_{#1}}%
            {\mathop{\vrule width 5pt height
3 pt depth -2.6pt
          \kern -6pt \intop}\nolimits_{#1}}}

\usepackage{color}

\definecolor{gr}{rgb}   {0.,   0.8,   0. } 
\definecolor{bl}{rgb}   {0.,   0.5,   1. } 
\definecolor{mg}{rgb}   {0.7,  0.,    0.7}

\begin{document}

\title[Estimates for operators beyond divergence form equations]{Square function and maximal function estimates for operators beyond divergence form equations}
\author[Andreas Ros\'en]{Andreas Ros\'en$\,^1$}
\thanks{$^1\,$Formerly Andreas Axelsson}
\address{Andreas Ros\'en, Matematiska institutionen, Link\"opings universitet, 581 83 Link\"oping, Sweden}
\email{andreas.rosen@liu.se}

\begin{abstract}
We prove square function estimates in $L_2$ for general operators of the form $B_1D_1+D_2B_2$,
where $D_i$ are partially elliptic constant coefficient homogeneous first order self-adjoint differential operators with orthogonal ranges,
and $B_i$ are bounded accretive multiplication operators, extending earlier estimates from the Kato 
square root problem to a wider class of operators.
The main novelty is that $B_1$ and $B_2$ are not assumed to be related in any way.
We show how these operators appear naturally from exterior differential 
systems with boundary data in $L_2$.
We also prove non-tangential maximal function estimates, where our proof needs only 
off-diagonal decay of resolvents in $L_2$, unlike earlier proofs which relied on interpolation and $L_p$
estimates.
\end{abstract}



\thanks{Supported by Grant 621-2011-3744 from the Swedish research council, VR}
\maketitle

\section{Introduction}  \label{sec:intro}

In this paper, we generalize the square function estimates from the Kato square root problem, to a wider class of operators on $L_2(\R^n; \C^N)$, $n,N\ge 1$.
Previously, estimates were known for perturbations of a homogeneous first order constant coefficients self-adjoint partial differential operator $D$, of the form $DB$ or $BD$, with $B$ being a bounded multiplication operator which is accretive on the range of $D$.
Let us first recall how such operators appear in connection with divergence form equations.
The celebrated Kato square root estimate
$$
  \|\sqrt{-\divv A \nabla} u\|_2\approx \|\nabla u\|_2
$$
for divergence form operators with general bounded accretive coefficients $A\in L_\infty(\R^n; \mL(\C^n))$,
were proved in one dimension by Coifman, McIntosh and Meyer~\cite{CMcM} and in full generality
by Auscher, Hofmann, Lacey, McIntosh and Tchamitchian~\cite{AHLMcT}.
It can be written $\|\sqrt{(BD)^2} [u, 0] \|_2\approx \|BD [u, 0]\|_2$, with 
$D= \begin{bmatrix} 0 & \divv \\ -\nabla & 0 \end{bmatrix}$ and $B= \begin{bmatrix} I & 0 \\ 0 & A \end{bmatrix}$, acting on vectors of dimension $N= 1+n$.
This estimate is in turn a consequence of square function estimates for the operator $BD$.
This approach to the Kato square root estimate was developed in \cite{AMcN, AKMc}.

Operators of the form $DB$ and $BD$ appear not only in connection with divergence form operators
on $\R^n$, but also in connection with divergence form equations on the half-space $\R^{1+n}_+:= \sett{(t,x)}{t>0, x\in\R^n}$,
with $L_2(\R^n)$ or $\dot H^1(\R^n)$ boundary data.
We recall the following approach to boundary value problems from \cite{AAM, AAH}.
Consider a divergence form equation
$$
   \divv_{t,x} A(t, x) \nabla_{t,x} u(t,x)=0, \qquad t>0, x\in\R^n,
$$
with coefficients $A= \begin{bmatrix} a & b \\ c & d \end{bmatrix}$, splitting $\C^{1+n}= \C\oplus \C^n$.
Write $e_0, e_1, \ldots, e_n$ for the standard basis in $\R^{1+n}$, with coordinates $x_0=t, x_1,\ldots x_n$,
and $f_\no:= e_0\cdot f$ for the normal component and $f_\ta:= f-f_\no e_0$ for the tangential part.

On the one hand, at the level of $\dot H^1(\R^n)$ boundary data $u|_{\R^n}$, we consider the conormal gradient
$$
  f= \begin{bmatrix} f_\no \\ f_\ta \end{bmatrix} := \begin{bmatrix} a\pd_t u + b\nabla_x u \\ \nabla_x u \end{bmatrix}.
$$
In terms of $f$, the divergence form equation is
$
  \pd_t f_\no + \divv_x ( c( a^{-1}f_\no-a^{-1}b f_\ta) + df_\ta)=0.
$
The conormal gradient $f$, with the inward conormal derivative as normal component $f_\no$, is in one-to-one correspondence with the potential $u$, modulo curl-freeness and constants.
Written in terms of $f$, the curl-free condition is 
$
  \pd_t f_\ta= \nabla_x( a^{-1}f_\no-a^{-1}b f_\ta )$,
  $\curl_x f_\ta=0$.
In vector notation this means that the divergence form equation for $u$ is equivalent to the 
vector valued ordinary differential equation 
$$
\pd_t f+ DBf=0
$$ 
for $f$
under the constraint $f\in\clos{\ran(D)}$ for each $t>0$,
with $D:= \begin{bmatrix} 0 & \divv_x \\ -\nabla_x & 0 \end{bmatrix}$ acting along $\R^n$ and $B:= \begin{bmatrix} a^{-1} & -a^{-1}b \\ ca^{-1} & d-ca^{-1}b \end{bmatrix}$ being a multiplication operator,
which turns out to be accretive if and only if $A$ is so.

On the other hand, at the level of $L_2(\R^n)$ boundary data $u|_{\R^n}$, we can write
$$
  f_\no= (A\nabla u)_\no =: \divv_x v_\ta,
$$
with a tangential vector field $v_\ta$, for each fixed $t$, assuming appropriate decay of $u$ at infinity, since $\int_{\R^n} f_\no dx=0$ by the divergence theorem.
Inserting this ansatz into the divergence form equation and commuting $\pd_t$ and $\divv_x$
yields $\divv_x(\pd_tv_\ta+ (A\nabla_{t,x}u)_\ta)=0$.
Since $v_\ta$ is only defined modulo tangential divergence free vector fields, we may choose it so that
$\pd_tv_\ta+ (A\nabla_{t,x}u)_\ta=0$.
In vector notation this means that the divergence form equation for $u$ is equivalent to the 
vector valued ordinary differential equation 
$$
\pd_t v+ BDv=0,
$$ 
for the vector field
$$
  v= \begin{bmatrix} v_\no \\ v_\ta \end{bmatrix} := \begin{bmatrix} -u \\ \divv_x^{-1}(a\pd_t u+ b\nabla_x u) \end{bmatrix}.
$$
One can view both $\pd_t f+ DBf=0$ and $\pd_t v+ BDv=0$ as generalized Cauchy--Riemann systems. 
In particular, the $n$ components of $v_\ta$ should be viewed as some generalized harmonic conjugate functions.

Estimates of operators of the form $DB$ or $BD$ are by now well understood, see \cite{AKMc, AAM, elAAM}.
The aim of this paper is to prove fundamental estimates for more general operators of the
form
$$
  B_1D_1+D_2B_2,
$$ 
which appear for example when, similar to above, writing a more general exterior differential system as a vector valued
ordinary differential equation in the variable transversal to the boundary.
See Section~\ref{sec:diffforms}.
We assume that $D_1D_2=0$ but, unlike earlier results \cite{elAAM}, not that $D_2B_2B_1D_1=0$.

We next formulate our results in detail. Consider four operators $D_1, D_2, B_1$ and $B_2$ 
acting in the Hilbert space $L_2(\R^n; \C^N)$ with norm $\|\cdot\|_2$, where $n,N\ge 1$.
We assume the following.
\begin{itemize}
\item The operators $D_1$ and $D_2$ are constant coefficient homogeneous first order differential operators which are self-adjoint and such that $\ran(D_2)\subset \nul(D_1)$.
Assume the partial ellipticity estimates $\|D_i f\|_2\gtrsim\|f\|_{\dot H^1(\R^n)}$ for all $f\in \ran(D_i)$, $i=1,2$.
\item The operators $B_1$ and $B_2$ are bounded multiplication operators $B_i: f(x)\mapsto B_i(x)f(x)$,
$x\in \R^n$, where $B_i(\cdot)\in L_\infty(\R^n;\mL(\C^N))$, $i=1,2$.
Assume the partial accretivity estimates
$\re (B_i f, f)\gtrsim \|f\|_2^2$ for all $f\in \ran(D_i)$, $i=1,2$.
\end{itemize}
Denote by $\omega_i:= \sup_{f\in\ran(D_i)\setminus\{0\}} |\arg(B_i f,f)|<\pi/2$ the angle of accretivity for $B_i$ on $\ran(D_i)$, $i=1,2$.
For $0\le \alpha<\pi/2$, define the closed sectors $S_{\alpha+}:= \sett{\lambda\in\C\setminus\{0\}}{|\arg\lambda|\le \alpha}\cup\{0\}$ and $S_{\alpha-}:= -S_{\alpha+}$, and bisectors $S_\alpha:= S_{\alpha+}\cup S_{\alpha-}$, as well as the corresponding open sectors/bisectors $S^o_{\alpha\pm}$, $S^o_{\alpha}$, being the interior of $S_{\alpha\pm}$, $S_{\alpha}$ respectively.

Denote by $\ran(\cdot)$, $\nul(\cdot)$ and $\dom(\cdot)$ the range, null space and domain of an operator.
Define the operator
$$
  T:= B_1D_1+ D_2B_2, \qquad \dom(T):= \sett{f\in \dom(D_1)}{B_2 f\in \dom(D_2)}.
$$
The definition of operators $\psi(tT)$, $\phi(tT)$ in the functional calculus of $T$, is found in Section~\ref{sec:ops}.
For a function $h$ on $\R^{1+n}_+$, define the ($L_2$ Whitney averaged) non-tangential maximal function 
$$
  \tN h(x):= \sup_{t>0} \left( \iint_{W(t,x)} |h(s,y)|^2 \frac{dsdy}{|W(t,x)|} \right)^{1/2},\qquad x\in \R^n,
$$
where $W(t,x)$ denotes a Whitney region around $(t,x)$, for example $W(t,x)= B(x,t)\times (t/2, 2t)$.

\begin{thm}  \label{thm:main}
Under the above hypothesis, $T$ is a closed and densely defined operator in $L_2(\R^n; \C^N)$, with spectrum
$\sigma(T)\subset S_\omega$, where $\omega:= \max(\omega_1,\omega_2)$,
and resolvent estimates
$\|(\lambda I-T)^{-1}\|\lesssim 1/\dist(\lambda, \sigma(T))$.

Moreover, the following estimates of holomorphic functions of the operator $T$ hold.
\begin{itemize}
\item We have square function estimates
$$
  \int_{\R^n}\int_0^\infty |\psi(tT)f(x)|^2 \frac {dtdx}t \lesssim \int_{\R^n} |f(x)|^2 dx, \qquad f\in \clos{\ran(T)},
$$
for any holomorphic symbol $\psi:S_\mu^o\to \C$, $\omega<\mu<\pi/2$, with estimates 
$|\psi(\lambda)|\lesssim \min(|\lambda|^s, |\lambda|^{-s})$ for some $s>0$.

If furthermore $\psi|_{S_{\omega+}}$ and $\psi|_{S_{\omega-}}$ are not identically zero, then the reverse square function estimates $\gtrsim$ hold for all $f\in \clos{\ran(T)}$.
\item
We have non-tangential maximal function estimates
$$
  \int_{\R^n} |\tN(\phi(tT)f)(x)|^2 dx \approx \int_{\R^n}|f(x)|^2 dx, \qquad f\in \clos{\ran(T)},
$$
for any holomorphic symbol $\phi:S_\mu^o\to \C$, $\omega<\mu<\pi/2$, 
with estimates $|\phi(\lambda)-1|\lesssim |\lambda|^s$ and $|\phi(\lambda)|\lesssim |\lambda|^{-s}$, for some $s>0$.
\end{itemize}
\end{thm}

The estimates in Theorem~\ref{thm:main} go back to the techniques from the solution of the Kato square root problem by Auscher, Hofmann, Lacey, McIntosh and Tchamitchian~\cite{AHLMcT}.
The connection between the Kato square root problem and square function estimates for first order differential operators was developed by Auscher, McIntosh and Nahmod~\cite{AMcN}.
More directly, both the square function and non-tangential maximal function estimates in
Theorem~\ref{thm:main} build on the author's joint work \cite{AAH} with Auscher and Hofmann.

So far, the Kato techniques have been applied to establish square function estimates for three main classes of first order differential operators.

\begin{enumerate}
\item Operators of the form $DB$ and $BD$, with $D$ being a self-adjoint constant coefficient homogeneous first order differential operator, and $B$ being a bounded multiplication operator which is accretive on $\ran(D)$. \label{eq:type1}
\item Operators of the form $\Gamma+ B_1\Gamma^*B_2$, with $\Gamma$ being a nilpotent (that is $\Gamma^2=0$) constant coefficient homogeneous first order differential operator, and $B_1$, $B_2$ being bounded multiplication operators such that 
$\Gamma^*B_2B_1\Gamma^*=0= \Gamma B_1B_2\Gamma$, which are accretive on $\ran(\Gamma^*)$ and $\ran(\Gamma)$ respectively. \label{eq:type2}
\item Operators of the form $B_1D_1+ D_2B_2$, with $D_1, D_2$ being self-adjoint constant coefficient homogeneous first order differential operators with orthogonal ranges, and $B_1$, $B_2$ being bounded  multiplication operators such that $D_2B_2B_1D_1=0$, which are accretive on $\ran(D_1)$ and $\ran(D_2)$ respectively. \label{eq:type3}
\end{enumerate}

Note that for all these classes of operators, there is essentially only one multiplication operator $B$ as we are considering a small generalization of the case $B_2=B_1^{-1}$. This is in constrast to Theorem~\ref{thm:main}, where the two multiplication operators $B_1$ and $B_2$ are independent of each other.

The main example of operators of type \eqref{eq:type2} are Hodge--Dirac operators, with $\Gamma$ being the exterior derivative acting on differential forms, see Section~\ref{sec:diffforms}. 
For operators of type \eqref{eq:type2}, square function estimates were proved in \cite[Thm. 2.7]{AKMc}.
By a simple operator theoretic argument, square function estimates for operators of type \eqref{eq:type1}
follow from such estimates of operators of type \eqref{eq:type2}, as shown in \cite[Thm. 3.1]{AKMc}.
In fact this argument can be reversed. It was shown in \cite[Sec. 8.1]{elAAM} that conversely square function estimates for operators of type \eqref{eq:type2} follow from such estimates for operators of type \eqref{eq:type1}.
Operators of type \eqref{eq:type3} are nothing but a direct sum of two operators of type \eqref{eq:type1},
and hence square function estimates are immediate as shown in \cite[Sec. 8.2]{elAAM}.
The type \eqref{eq:type3} operators first appeared, in disguise, in the work by 
Auscher, Axelsson and Hofmann~\cite{AAH}, where boundary value problems for Dirac equations of the form $(\Gamma+ B_1\Gamma^*B_2)f=0$, $\Gamma$ being the exterior derivative, were studied. Similarly to Section~\ref{sec:diffforms} here, it was shown in \cite[Sec. 8.3]{elAAM}, that solving for the $t$-derivatives, this equation can be written $(\pd_t + (BD_1+ D_2B^{-1})) Uf=0$, under suitable similarity transformation $U$. It should be noted that in \cite{AAH}, $B_1$ and $B_2$ were related in exactly the way so that the associated operator $BD_1+ D_2B^{-1}$ is of type \eqref{eq:type3} and not of the more general form $T$ considered in Theorem~\ref{thm:main}, where the two multiplication operators are independent.

Coming to the non-tangential maximal function estimates in Theorem~\ref{thm:main}, these build 
on the estimates by Auscher, Axelsson and Hofmann in \cite[Prop. 2.56]{AAH}.
Although set in the framework with Dirac equations, what was actually proved there was non-tangential maximal function estimates for operators of type \eqref{eq:type1}, in the special case when
the differential operator is of the form $D= \begin{bmatrix} 0 & \divv \\ -\nabla & 0 \end{bmatrix}$.
Even for operators of type \eqref{eq:type1} with more general $D$, the non-tangential maximal function estimate in Theorem~\ref{thm:main} are new. Also the non-tangential maximal function estimates for general operators of type
\eqref{eq:type2} and \eqref{eq:type3}, which is a special case of Theorem~\ref{thm:main}, are new.

The key idea in the proof of the square function estimates in Theorem~\ref{thm:main} is to use a 
splitting of $L_2$ adapted to the operators $B_1D_1$ and $D_2B_2$, in which the operator $B_1D_1+D_2B_2$ is triangular due to the assumption $D_1D_2=0$.
The proof of the non-tangential maximal function estimates in Theorem~\ref{thm:main} is much inspired
by the proof of \cite[Prop. 2.56]{AAH}. The main difference is that our proof here is a pure $L_2$ proof, in that it only requires $L_2$ off-diagonal decay of resolvents.
In \cite[Prop. 2.56]{AAH}, interpolation theory to prove $L_p$ off-diagonal decay, $p\approx 2$, was needed.
Another novelty is a Caccioppoli type estimate, Lemma~\ref{lem:cacc}, for operators beyond divergence form equations.

The outline of this paper is as follows. In Section~\ref{sec:diffforms}, we show how operators of the form 
$B_1D_1+D_2B_2$ arise naturally in connection with exterior differential systems systems in $\R^{1+n}_+$ for differential forms. The special case of one-forms, that is vector fields, amounts to divergence form equations. 
In Section~\ref{sec:ops} we prove the resolvent estimates for the operator $T$, in Section~\ref{sec:sqfcn} we prove the square function estimates for the operator $T$, and finally in Section~\ref{sec:maxfcn}
we prove the non-tangential maximal function estimates for the operator $T$.
The (roadmap to the) proof of Theorem~\ref{thm:main} is in Section~\ref{sec:maxfcn}.

\section{Resolvent estimates}    \label{sec:ops}

In this section, we establish the basic operator theoretical properties of the operator
$T= B_1D_1+D_2B_2$.
A fundamental observation for the unperturbed operator $D_1+D_2$ is the orthogonal splitting
$$
  L_2(\R^n;\C^N)= (\nul(D_1)\cap \nul(D_2))\oplus \clos{\ran(D_1)}\oplus \clos{\ran(D_2)}=: \mH_0\oplus \mH_1\oplus \mH_2.
$$
The natural perturbation of this splitting which is adapted to the operator $T$ is 
$$
 L_2(\R^n;\C^N)= (\nul(D_1)\cap \nul(D_2B_2))\oplus \clos{\ran(B_1D_1)}\oplus \clos{\ran(D_2)}.
$$

\begin{prop}    \label{prop:split}
We have a topological (but in general not orthogonal) splitting
$$
  L_2(\R^n;\C^N)= \mH_0^{B_2}\oplus B_1 \mH_1\oplus \mH_2,
$$ 
where $\mH_0^{B_2}:= (\mH_0\oplus \mH_2)\cap B_2^{-1}(\mH_0\oplus \mH_1)$, $B_2^{-1} V:= \sett{f\in L_2}{B_2 f\in V}$ and $B_1 V:= \sett{B_1f}{f\in V}$.
\end{prop}

\begin{proof}
  We observe that we have two topological splittings 
$$
  L_2= B_1\mH_1\oplus (\mH_0\oplus \mH_2)
$$  
  and $L_2= B_2^*\mH_2\oplus (\mH_0\oplus \mH_1)$,
since $B_1$ is accretive on $\mH_1$ and $B_2$, and hence $B_2^*$, is accretive on $\mH_2$.
See for example \cite[Prop. 3.3]{AAM}.
Taking orthogonal complements in the second splitting, we obtain a third topological splitting
$$
  L_2= B_2^{-1}(\mH_0\oplus \mH_1)\oplus \mH_2,
$$
since $B_2^{-1}(\mH_0\oplus \mH_1)= (B_2^*\mH_2)^\perp$ and $\mH_2= (\mH_0\oplus \mH_1)^\perp$.

The result is now a consequence of Lemma~\ref{lem:splitbanach} below,
with $X_1= B_1\mH_1$, $X_2= \mH_0\oplus \mH_2$, $X_3= B_2^{-1}(\mH_0\oplus \mH_1)$ and $X_4= \mH_2$.
\end{proof}

\begin{lem}  \label{lem:splitbanach} 
Assume that a Banach space $X$ splits topologically in two ways
$$
  X= X_1\oplus X_2= X_3\oplus X_4,
$$
into closed subspaces such that $X_4\subset X_2$.
Then $X$ splits topologically into three closed subspaces
$$
  X= X_1\oplus (X_2\cap X_3)\oplus X_4.
$$
\end{lem}

\begin{proof}
It is straightforward to verify that these three subspaces are closed and intersect pair wise only at $0$. Also, given $x\in X$, we can write $x= x_1+x_2$ and $x_2= x_3+x_4$ with $x_i\in X_i$, $i=1, 2, 3, 4$.
We have $x=x_1+x_3+x_4$, with $x_3= x_2-x_4\in X_3\cap X_2$, so the three subspaces span $X$.
\end{proof}

\begin{prop}   \label{prop:nulrang}
For the operator $T$, the null space is $\nul(T)= \mH_0^{B_2}$, the range is $\ran(T)=B_1\ran(D_1)+\ran(D_2)$ and the domain is
$$\dom(T)=  \sett{u_0+u_1+u_2\in \mH_0^{B_2}\oplus B_1 \mH_1\oplus \mH_2}{u_1\in\dom(D_1),  u_1+u_2\in \dom(D_2B_2)}.
$$
\end{prop}

\begin{proof}
  For the null space, we note that $f\in\nul(T)$ if and only if $B_1D_1f= -D_2B_2f$
  Since $B_1\mH_1\cap \mH_2=\{0\}$, this is equivalent to $f\in \nul(D_1)\cap B_2^{-1}\nul(D_2)$.
  
 Clearly $\ran(T)\subset B_1\ran(D_1)+ \ran(D_2)$. For the converse implication, assume that
 $f_1= B_1D_1 u_1\in B_1\ran(D_1)$ and $f_2= D_2B_2u_2\in \ran(D_2)= \ran(D_2B_2)$. 
 Write $u_1= u_1^1+ u_1^2\in \mH_2\oplus B_2^{-1}(\mH_0\oplus \mH_1)$ and 
 $u_2= u_2^1+ u_2^2\in \mH_2\oplus B_2^{-1}(\mH_0\oplus \mH_1)$.
 Then $(B_1D_1+D_2B_2)(u_1^2+u_2^1)= B_1D_1u_1+ D_2B_2 u_2= f_1+f_2$,
 so $\ran(T)= B_1\ran(D_1) + \ran(D_2)$.

 The result for the domain follows from the facts that $\mH_0^{B_2}\subset \dom (D_1)\cap \dom (D_2B_2)$ and $\mH_2\subset \dom(D_1)$.
\end{proof}

We now express the resolvents of $T$ in terms of the resolvents 
$$
  R^1_t:= (I+itB_1D_1)^{-1}\qquad\text{and}\quad R^2_t:= (I+itD_2B_2)^{-1}
$$
of $B_1D_1$ and $D_2B_2$.
It is known that $\sigma(B_1D_1)\subset S_{\omega_1}\cup\{0\}$ with resolvent estimates
$\|R^1_{t}\|\lesssim 1/(|t|\dist(i/t, S_{\omega_1}))$, and that
$\sigma(D_2B_2)\subset S_{\omega_2}\cup\{0\}$ with resolvent estimates
$\|R^2_{t}\|\lesssim 1/(|t|\dist(i/t, S_{\omega_2}))$.
See for example \cite[Prop. 3.3]{AAM}.

\begin{prop}   \label{prop:blockresolvent}
The operator $T$ is closed and densely defined in $L_2(\R^n;\C^N)$.
  The spectrum is contained in the bisector $S_\omega\cup\{0\}$, $\omega= \max(\omega_1, \omega_2)$, and in the splitting
  $L_2(\R^n;\C^N)= \mH_0^{B_2}\oplus B_1 \mH_1\oplus \mH_2$, the resolvent
  has the expression
$$
  (I+itT)^{-1}=
  \begin{bmatrix} I & 0 & 0 \\ 0 & R^1_t & 0 \\ 0 & (R_t^2-I)R_t^1 & R_t^2 \end{bmatrix},
  \qquad i/t\notin S_\omega\cup\{0\},
$$
with estimates $\|(I+itT)^{-1}\|\lesssim 1/(|t|\dist(i/t, S_\omega))$.
\end{prop}

\begin{proof}
 Consider $(I+itT)u=f$ with $u\in\dom(T)$, and write $u=u_0+u_1+u_2$ and $f=f_0+f_1+f_2$ in the 
 splitting from Proposition~\ref{prop:split}.
 Then
\begin{gather*}
  u_0=f_0, \qquad u_1+it B_1D_1u_1=f_1, \\
  u_2+ itD_2B_2(u_1+u_2)=f_2.
\end{gather*}
Solving for $u$, we equivalently have
\begin{gather*}
  u_0=f_0, \qquad u_1= (I+itB_1D_1)^{-1} f_1, \\
  u_2= (I+itD_2B_2)^{-1}f_2-itD_2B_2(I+itD_2B_2)^{-1}(I+itB_1D_1)^{-1}f_1.
\end{gather*}
This shows that $I+itT$ is injective with the stated resolvent expression and estimate.

To show surjectivity, given $f\in L_2(\R^n;\C^N)$, define 
$u:= f_0+ (I+itD_2B_2)^{-1}f_2+(I+itD_2B_2)^{-1}(I+itB_1D_1)^{-1}f_1$.
Then reversing the above calculation, shows that $u\in\dom(T)$ and $(I+itT)u=f$.
It follows that $I+itT$ is surjective and that $T$ is a closed operator. 
That $T$ is densely defined, follows from the fact that $\dom(T)$ contains the dense subspace
$\dom(D_2B_2)\cap\ran(D_2B_2)$.
\end{proof}

\begin{prop}   \label{prop:duality}
The adjoint of $T$ is
$T^*= B_2^*D_2+ D_1B_1^*$,
with domain $\dom(T^*):= \sett{f\in \dom(D_2)}{B_1^* f\in \dom(D_1)}$.
\end{prop}

\begin{proof}
It suffices to show that if 
\begin{equation}   \label{eq:fcnlcont}
\dom(T)\to \C: u\mapsto (Tu,v)
\end{equation}
is $L_2$ continuous, then $v\in \dom(D_2)\cap \dom(D_1B_1^*)$.
The splitting for $v$ analogous to Proposition~\ref{prop:split} for $u$, is
$$
  v= v_0+v_1+v_2\in ((\mH_0\oplus \mH_1)\cap (B_1^*)^{-1}(\mH_0\oplus \mH_2))\oplus B_2^*\mH_2\oplus \mH_1.
$$
We need to show $v_1\in \dom(D_2)$ and $v_1+v_2\in \dom(D_1B_1^*)$.
To this end, let $u=u_2\in \mH_2\cap\dom(D_2B_2)\subset \dom(T)$ in \eqref{eq:fcnlcont}.
Then
$$
  |(D_2B_2 u_2, v_1)|= |(Tu_2, v)|\lesssim \|u_2\|_2.
$$
It follows that $v_1\in \dom(D_2)$, since $(D_2B_2)^*= B_2^*D_2$.
Therefore, for general $u\in\dom(T)$, we have
$$
  (Tu,v)= (B_1D_1u_1, v)+ (u,B_2^*D_2 v_1),
$$
so $|(B_1 D_1 u_1, v_1+v_2)|\lesssim \|u\|_2$.
Since $(B_1D_1)^*= D_1B_1^*$, it follows that $v_1+v_2\in\dom(D_1B_1^*)$.
\end{proof}

We end this section with a short discussion of the definition of the functional calculus of bisectorial operators. 
For further details see \cite{ADMc}, where the corresponding theory for sectorial operators is readily adapted to bisectorial operators.

Given a bisectorial operator $T$ in a Hilbert space $\mH$, that is a closed and densely defined operator $T$ with $\sigma(T)\subset S_\omega$ for some $\omega<\pi/2$ and resolvent bounds 
$$
  \|(\lambda I-T)^{-1}\|\lesssim 1/\dist(\lambda, S_\omega),
$$
there is a natural definition of $\phi(T)$ for any rational function $\phi(\lambda)$ which is bounded and without poles in $S_\omega$. Useful such symbols in this paper are for example $1/(1+t^2\lambda^2)$ and $t\lambda/(1+t^2\lambda^2)$, with scale parameter $t>0$.

If $T$ is not injective, then there is a topological splitting $\mH= \mH_0\oplus \mH_1$,
with $\mH_0= \nul(T)$ and $\mH_1= \clos{\ran(T)}$.
Indeed, 
$$
   I = (I+itT)^{-1}+ itT(I+itT)^{-1},
$$
where the two terms converge strongly to the projections onto $\mH_0$ and $\mH_1$ respectively as $t\to\infty$.
Thus $T= 0\oplus T_1$ and $(\lambda I- T)^{-1}= \lambda^{-1}I \oplus (\lambda I-T_1)^{-1}$, where $T_1:= T|_{\mH_1}$ is an injective bisectorial operator in $\mH_1$.

With the Dunford integral 
$$
  \psi(T):=\psi(0)I_{\mH_0}\oplus \frac 1{2\pi i}\int_\gamma \psi(\lambda)(\lambda I_{\mH_1}- T_1)^{-1} d\lambda,
$$
the functional calculus is extended to all symbols $\psi: S_\mu^o\cup\{0\}\to\C$ which are holomorphic 
on an open bisector $S_\mu^o\supset S_\omega\setminus\{0\}$, with estimates $|\psi(\lambda)|\lesssim \min(|\lambda|^s, |\lambda|^{-s})$ for some $s>0$, so that the integral is convergent in the operator norm on $\mH_1$. Here the curve $\gamma= \sett{te^{\pm i\theta}}{t\in \R}$, $\omega<\theta<\mu$, is oriented counter clockwise around $S_\omega$.

To obtain $\phi(T)$ as bounded operators on $\mH$ for general bounded holomorphic symbols $\phi$, without decay at $0$ and $\infty$, 
square function estimates
$$
  \int_0^\infty \|\psi(tT)f\|_{\mH_1}^2\frac {dt}t\approx \|f\|_{\mH_1}^2,\qquad f\in \mH_1,
$$
are required. Usually, it suffices to show estimates $\lesssim$, as estimates $\gtrsim$ for $T$ follow from
estimates $\lesssim$ for $T^*$. Another basic result concerning square function estimates, is that
if such hold for one symbol $\psi$, then they hold for all $\tilde \psi$ such that $|\tilde\psi(\lambda)|\lesssim \min(|\lambda|^s, |\lambda|^{-s})$ for some $s>0$, and $\tilde \psi|_{S_\omega\pm}\ne 0$.

Given such square function estimates, it follows that
$\psi_n(T)$ are uniformly bounded and converges strongly in $\mL(\mH)$ whenever $\sup_{n\in\Z_+, \lambda\in S_\mu^o}|\psi_n(\lambda)|<\infty$ and $\psi_n(\lambda)$ converges for each $\lambda\in S_\mu^o\cup\{0\}$.
Through such a limiting argument, we construct a bounded homomorphism
$$
  \phi\mapsto \phi(T),
$$
taking bounded symbols $\phi: S_\mu^o\cup\{0\}\to \C$ which are holomorphic on $S_\mu^o$ to bounded linear operators $\phi(T)$ on $\mH$.

\section{Applications to exterior differential systems}   \label{sec:diffforms}

In this section, we show how operators of the form $B_1D_1+D_2B_2$ appear in connection with 
exterior differential systems for differential form.
We first fix notation.
Instead of writing $\{dx_0, dx_1,\ldots, dx_n\}$ for the basis one-forms, we shall keep the notation
$\{e_0,e_1,\ldots, e_n\}$ from Section~\ref{sec:intro} for the basis vectors, and we use the terminology $k$-vector field instead of $k$-form, in the euclidean space $\R^{1+n}$.

The space of $k$-vectors in $\R^{1+n}$ we define to be the $\binom {1+n}{k}$ dimensional complex linear space
$$
  \wedge^k \R^{1+n}:= \text{span}_\C\sett{e_{s_1}\wedg\ldots \wedg e_{s_k}}{0\le s_1<s_2<\ldots s_k\le n},
  \qquad k=2,\ldots, n+1,
$$
and we let $\wedge^0\R^{1+n}:= \C$, $\wedge^1\R^{1+n}:=\C^{1+n}$ and  $\wedge^k \R^{1+n}:= \{0\}$ if $k\notin \{0,1,\ldots, n+1\}$.

Given a vector $v= \sum_{j=0}^n v_je_j\in \wedge^1\R^{1+n}$ and a $k$-vector
$w=\sum_{0\le s_1<\ldots s_k\le n} w_s e_s\in \wedge^k \R^{1+n}$, writing $s= \{s_1, \ldots, s_k\}$ and $e_s:= e_{s_1}\wedg \ldots\wedg e_{s_k}$, we have in particular the exterior product
$v\wedg w\in \wedge^{k+1}\R^{1+n}$ and the (left) interior product $v\lctr w\in \wedge^{k-1}\R^{1+n}$
defined bilinearly using 
$$
  e_j\wedg e_s:= \begin{cases} \epsilon(j,s) e_{\{j\}\cup s}, & j\notin s, \\
  0, & j\in s, \end{cases}\qquad
  e_j\lctr e_s:= \begin{cases} 0, & j\notin s, \\
  \epsilon(j,s\setminus\{j\}) e_{s\setminus \{j\}}, & j\in s,
  \end{cases}
$$
where the permutation sign is $\epsilon(j,s):= (-1)^{|\sett{s_i}{j>s_i}|}$.
Defining inner products on $\wedge^k\R^{1+n}$, $k=0,1,\ldots, n+1$, so that the standard bases above are ON-bases,
we have that
$$
  v\wedg(\cdot): \wedge^k\R^{1+n}\to \wedge^{k+1}\R^{1+n}\qquad \text{and}\qquad
  v\lctr(\cdot): \wedge^{k+1}\R^{1+n}\to \wedge^{k}\R^{1+n},
$$
are adjoint multiplication operators if $v$ is a vector, that is $v\in \wedge^1\R^{1+n}$, with real coefficients.
The corresponding differential operators are the exterior derivative operator
$$
  \nabla_{t,x}\wedg f:= \sum_{j=0}^n e_j\wedg \pd_j f,
$$
mapping $k$-vector fields $f:\R^{1+n}\to \wedge^k\R^{1+n}$
to $k+1$-vector fields, and the interior derivative operator
$$
  \nabla_{t,x}\lctr g:= \sum_{j=0}^n e_j\lctr \pd_j g,
$$
mapping $k+1$-vector fields $g:\R^{1+n}\to \wedge^{k+1}\R^{1+n}$
to $k$-vector fields.
As special cases of these operators, we have the gradient and curl, being the exterior derivative acting on scalars and vectors ($k=0$ and $k=1$ respectively), and the divergence being the interior derivative acting on vectors.
We also note the duality $\iint(\nabla_{t,x}\wedg f, g)dtdx= -\iint (f,\nabla_{t,x}\lctr g)dtdx$ for compactly supported fields.

The basic exterior differential system in $\R^{1+n}$ that we want to consider is
\begin{equation}   \label{eq:extssytemunpert}
\begin{cases}
  \nabla_{t,x}\lctr \tilde f_{k+1}= \nabla_{t,x}\wedg \tilde f_{k-1}, \\
  \nabla_{t,x} \wedg \tilde f_{k+1}= 0=  \nabla_{t,x} \lctr \tilde f_{k-1},
\end{cases}
\end{equation}
for a $k+1$-vector field $\tilde f_{k+1}$ and a $k-1$-vector field $\tilde f_{k-1}$.
Two important special cases are the following.
If $k=0$, then the system reads $\divv_{t,x}\tilde f_1=0= \curl_{t,x} \tilde f_1$, since $\tilde f_{-1}=0$. 
This is nothing but the Laplace equation, written for the gradient as in Section~\ref{sec:intro}.
If $k=1$, then the system reads $\nabla_{t,x}\lctr \tilde f_2=\nabla_{t,x} \tilde f_0$, $\nabla_{t,x}\wedg f_2=0$.
This equation is the Stokes' system of linearized hydrostatics, written for the vorticity $\tilde f_2$ and the the pressure $\tilde f_0$.

Consider next a bilipschitz map $\rho:\R^{1+n}_+\to \Omega\subset \R^{1+n}$, and
the system \eqref{eq:extssytemunpert} in $\Omega$.
We want to pull back this system of equations to $\R^{1+n}_+$, and recall therefore the following 
facts from differential geometry.
At a fixed point in $\R^{1+n}_+$, denote by $\linj\rho$ the Jacobian matrix of all partial derivatives of $\rho$. Extend this linear map as a $\wedg$-homomorphism to $\wedge^k\R^{1+n}$, letting
$$
  \linj \rho(e_{s_1}\wedg\ldots\wedg e_{s_k}):= (\linj \rho e_{s_1})\wedg\ldots\wedg(\linj\rho  e_{s_k}).
$$
Given a $k$-vector field $f:\Omega\to \wedge^k\R^{1+n}$, we define the pullback of $f$ by $\rho$ to be the 
$k$-vector field
$$
  \rho^* f(t,x):= \linj\rho_{(t,x)}^* (f(\rho(t,x))) 
$$
in $\R^{1+n}_+$, where $\linj\rho_{(t,x)}^*$ is the adjoint of the Jacobian matrix at $(t,x)$.
A fundamental well known result is that
\begin{equation}   \label{eq:commext}
  \nabla_{t,x}\wedg (\rho^* f)= \rho^*(\nabla_{t,x}\wedg f).
\end{equation}
Less commonly used is the equivalent dual result that
\begin{equation}    \label{eq:commint}
  \nabla_{t,x}\lctr (J_\rho \rho^{-1}_* g)= J_\rho \rho^{-1}_*(\nabla_{t,x}\lctr g),
\end{equation}
where $J_\rho$ is the Jacobian determinant of $\rho$ and 
$$
  \rho_*^{-1} g(t,x):= \linj\rho_{(t,x)} (g(\rho(t,x)))
$$
is the push forward of $g$ by $\rho^{-1}$.
Applying \eqref{eq:commext} and \eqref{eq:commint}, we find that 
\eqref{eq:extssytemunpert} in $\Omega$ is equivalent to 
\begin{equation}  \label{eq:extssytem}
\begin{cases}
  \nabla_{t,x}\lctr (A_{k+1}(t,x) f_{k+1})= A_k(t,x) (\nabla_{t,x}\wedg f_{k-1}), \\
  \nabla_{t,x} \wedg f_{k+1}= 0=  \nabla_{t,x} \lctr (A_{k-1}(t,x) f_{k-1})
\end{cases}
\end{equation}
in $\R^{1+n}_+$, where $f_j:= \rho^*\tilde f_j$ and $A_j:= J_\rho(\rho^*\rho_*)^{-1}$ is the Jacobian determinant times the
inverse of the metric tensor $G= \rho^*\rho_*$, extended as a $\wedg$-homomorphism to $\wedge^j\R^{1+n}$.

We now show, analogous to the case $k=0$ in the introduction, how \eqref{eq:extssytem} is equivalent to a vector valued ordinary differential equation $\pd_t f_t+ Tf_t=0$,
for general bounded measurable and accretive coefficients $A_j(t,x)\in \mL(\wedge^j\R^{1+n})$, $j=k-1, k, k+1$,
with an infinitesimal generator $T$ of the form $T= B_1D_1+D_2B_2$. 
We use the natural identifications
\begin{align*}
   \wedge^k\R^n\oplus \wedge^{k+1}\R^n=\wedge^{k+1}\R^{1+n} &: \lf_k\oplus \lf_{k+1}\approx e_0\wedg \lf_k+ \lf_{k+1}= f_{k+1}, \\
   \wedge^{k-2}\R^n\oplus \wedge^{k-1}\R^n=\wedge^{k-1}\R^{1+n} &: \lf_{k-2}\oplus \lf_{k-1}\approx e_0\wedg \lf_{k-2}+ \lf_{k-1}= f_{k-1},
\end{align*}
with corresponding splittings of the coefficient matrices so that
\begin{align*}
  A_{k+1}f_{k+1} &= e_0\wedg( a_{k+1}\lf_k + b_{k+1}\lf_{k+1})+ ( c_{k+1}\lf_k + d_{k+1}\lf_{k+1}), \\
  A_{k-1}f_{k-1} &= e_0\wedg( a_{k-1}\lf_{k-2} + b_{k-1}\lf_{k-1})+ ( c_{k-1}\lf_{k-2} + d_{k-1}\lf_{k-1}), 
\end{align*}
and similarly for $A_k$.

Let $\mH_{\wedg}^j$ denote the closure of the range of $\nabla_x\wedg(\cdot): L_2(\R^n; \wedge^{j-1}\R^n)\to L_2(\R^n; \wedge^{j}\R^n)$, and let $ \mH_{\lctr}^j$ be the closure of the range of $\nabla_x\lctr(\cdot): L_2(\R^n; \wedge^{j+1}\R^n)\to L_2(\R^n; \wedge^{j}\R^n)$. 
Fundamental results are that $\mH_{\wedg}^j$ is the null space of $\nabla_x\wedg(\cdot): L_2(\R^n; \wedge^{j}\R^n)\to L_2(\R^n; \wedge^{j+1}\R^n)$, $\mH_{\lctr}^j$ is the nullspace of $\nabla_x\lctr(\cdot): L_2(\R^n; \wedge^{j}\R^n)\to L_2(\R^n; \wedge^{j-1}\R^n)$, and we have an orthogonal Hodge splitting
$$
  L_2(\R^n;\wedge^j\R^n)= \mH_{\wedg}^j\oplus \mH_{\lctr}^j.
$$
\begin{prop}   \label{prop:extsyst}
Assume that $A_j\in L_\infty(\R^n; \mL(\wedge^{j-1}\R^n\oplus \wedge^j\R^n))$ are $t$-independent
and accretive on $\mH_{\lctr}^{j-1}\oplus \mH_{\wedg}^j$, $j=k-1, k, k+1$. 
  Define the $\wedge^{k-2}\R^n\oplus \wedge^{k-1}\R^n\oplus \wedge^{k}\R^n\oplus \wedge^{k+1}\R^n$ valued function $\hf$ by
$$
  \hf =
  \begin{bmatrix}
      \hf_{k-2} \\ \hf_{k-1} \\  \hf_{k} \\ \hf_{k+1} 
  \end{bmatrix}  
   :=
  \begin{bmatrix}
     a_{k-1}\lf_{k-2}+ b_{k-1} \lf_{k-1} \\ \lf_{k-1} \\  a_{k+1}\lf_{k}+ b_{k+1} \lf_{k+1} \\ \lf_{k+1} 
  \end{bmatrix}.
$$
Then the exterior differential system \eqref{eq:extssytem} for the $\wedge^{k-1}\R^{1+n}\oplus \wedge^{k+1}\R^{1+n}$ valued function $f_{k-1}\oplus f_{k+1}$ is equivalent to the 
vector valued ordinary differential equation
$$
  \pd_t \hf + (B_1D_1+D_2B_2)\hf=0,
$$
together with the constraint $\hf\in \clos{\ran(B_1D_1+D_2B_2)}$ for each $t$,
where 
\begin{align*}
B_1&:= \begin{bmatrix}
  0 & 0 & 0 & 0 \\
  0 & a_k^{-1} & -a_k^{-1}b_k & 0 \\
  0 & c_k a_k^{-1} & d_k-c_ka_k^{-1}b_k & 0 \\
  0 & 0 & 0 & 0
\end{bmatrix}, \\
D_1&:=
\begin{bmatrix}
  0 & 0 & 0 & 0 \\
  0 & 0 & \nabla_x\lctr(\cdot) & 0 \\
  0 & -\nabla_x\wedg(\cdot) & 0 & 0  \\
  0 & 0 & 0 & 0 
\end{bmatrix}, \\
D_2&:=
\begin{bmatrix}
  0 & \nabla_x\lctr(\cdot) & 0 & 0 \\
  -\nabla_x\wedg(\cdot) & 0 & 0 & 0 \\
  0 & 0 &   0 & \nabla_x\lctr(\cdot)  \\
  0 & 0 & -\nabla_x\wedg(\cdot) & 0 \\
\end{bmatrix} \quad\text{and}\\
B_2&:=
\begin{bmatrix}
  a_{k-1}^{-1} & -a_{k-1}^{-1}b_{k-1} & 0 & 0 \\
  c_{k-1} a_{k-1}^{-1} & d_{k-1}-c_{k-1}a_{k-1}^{-1}b_{k-1} & 0 & 0 \\
  0 & 0 & a_{k+1}^{-1} & -a_{k+1}^{-1}b_{k+1} \\
  0 & 0 & c_{k+1} a_{k+1}^{-1} & d_{k+1}-c_{k+1}a_{k+1}^{-1}b_{k+1}  
\end{bmatrix}
\end{align*}
satisfy the hypothesis in Theorem~\ref{thm:main}.
\end{prop}

\begin{proof}
  The equation $\nabla_{t,x} \wedg f_{k+1}= 0$ is equivalent to
$$
  \begin{cases}
    \pd_t \lf_{k+1} - \nabla_x\wedg \lf_k=0, \\
    \nabla_x\wedg \lf_{k+1}=0.
  \end{cases}
$$
  The equation $\nabla_{t,x} \lctr (A_{k-1} f_{k-1})= 0$ is equivalent to
$$
  \begin{cases}
    \nabla_x\lctr (a_{k-1} \lf_{k-2}+ b_{k-1}\lf_{k-1})=0, \\
    \pd_t ( a_{k-1} \lf_{k-2}+ b_{k-1} \lf_{k-1}) + \nabla_x\lctr(c_{k-1} \lf_{k-2}+ d_{k-1} \lf_{k-1})=0.
  \end{cases}
$$
  The equation $\nabla_{t,x} \lctr (A_{k+1} f_{k+1})= A_k(\nabla_{t,x}\wedg f_{k-1})$ is equivalent to
$$
  \begin{cases}
    -\nabla_x\lctr (a_{k+1} \lf_{k}+ b_{k+1}\lf_{k+1})= a_k(\pd_t \lf_{k-1}-\nabla_x\wedg \lf_{k-2})+ b_k(\nabla_x\wedg \lf_{k-1}), \\
    \pd_t ( a_{k+1} \lf_{k}+ b_{k+1} \lf_{k+1}) + \nabla_x\lctr(c_{k+1} \lf_{k}+ d_{k+1} \lf_{k+1}) 
    \\ \qquad\qquad=
    c_k(\pd_t \lf_{k-1}-\nabla_x\wedg \lf_{k-2})+ d_k(\nabla_x\wedg \lf_{k-1}).
  \end{cases}
$$
Written in terms of $\hf$, the four evolution equations are
$$
  \begin{cases}
    \pd_t  \hf_{k-2} + \nabla_x\lctr(c_{k-1} \lf_{k-2}+ d_{k-1} \lf_{k-1})=0, \\
    \pd_t \hf_{k-1}-\nabla_x\wedg \lf_{k-2}+ a_k^{-1}b_k(\nabla_x\wedg \lf_{k-1}) +
    a_k^{-1}\nabla_x\lctr (a_{k+1} \lf_{k}+ b_{k+1}\lf_{k+1})=0, \\
    \pd_t  \hf_{k} + \nabla_x\lctr(c_{k+1} \lf_{k}+ d_{k+1} \lf_{k+1}) 
    -
    c_k(\pd_t \hf_{k-1}-\nabla_x\wedg \lf_{k-2})- d_k(\nabla_x\wedg \lf_{k-1})=0, \\
    \pd_t \hf_{k+1} - \nabla_x\wedg \lf_k=0, 
  \end{cases}
$$
and the remaining two equations give the constraints $\nabla_x\wedg \hf_{k+1}=0= \nabla_x\lctr \hf_{k-2}$.
We next write the evolution equations in terms of $\hf$, using 
$\lf_{k-2}= a_{k-1}^{-1}(\hf_{k-2}- b_{k-1}\hf_{k-1})$ and
$\lf_{k}= a_{k+1}^{-1}(\hf_{k}- b_{k+1}\hf_{k+1})$.
The tangential derivatives in the evolution equations that appear with coefficients to the right are
$$
\begin{bmatrix}
  \nabla_x\lctr (c_{k-1}a_{k-1}^{-1} (\hf_{k-2}-b_{k-1}\hf_{k-1})+ d_{k-1}\hf_{k-1}) \\
  -\nabla_x\wedg( a_{k-1}^{-1}(\hf_{k-2}- b_{k-1}\hf_{k-1})) \\
  \nabla_x\lctr (c_{k+1}a_{k+1}^{-1} (\hf_{k}-b_{k+1}\hf_{k+1})+ d_{k+1}\hf_{k+1}) \\
    -\nabla_x\wedg( a_{k+1}^{-1}(\hf_{k}- b_{k+1}\hf_{k+1})) 
\end{bmatrix}
=D_2B_2\hf.
$$
The tangential derivatives in the evolution equations that appear with coefficients to the left are
$$
\begin{bmatrix}
 0 \\
 a_k^{-1} b_k \nabla_x\wedg \hf_{k-1}+ a_k^{-1} \nabla_x\lctr \hf_k \\
 c_k( a_k^{-1}b_k\nabla_x\wedg\hf_{k-1}+ a_k^{-1}\nabla_x\lctr \hf_k )- d_k\nabla_x\wedg \hf_{k-1} \\
 0  
\end{bmatrix} \\
=B_1D_1\hf.
$$
This shows that the evolution equation for $\hf$ is $\pd_t \hf+ (B_1D_1+D_2B_2)\hf=0$.
To show that the constraint $\nabla_x\wedg \hf_{k+1}=0= \nabla_x\lctr \hf_{k-2}$ is equivalent
to $f\in \clos{\ran(B_1D_1+D_2B_2)}$, we note that 
\begin{multline*}
\clos{\ran(B_1D_1+D_2B_2)}=B_1\clos{\ran(D_1)}\oplus \clos{\ran(D_2)}\\=
B_1(\mH_{\wedg}^k\oplus \mH_{\lctr}^{k-1})\oplus \big((\mH_{\lctr}^k\oplus \mH_{\wedg}^{k-1})\oplus (\mH_{\wedg}^{k+1}\oplus\mH_{\lctr}^{k-2})\big)\\=L_2(\R^n; \wedge^{k-1}\R^n\oplus \wedge^k\R^{n})\oplus (\mH_{\wedg}^{k+1}\oplus\mH_{\lctr}^{k-2}),
\end{multline*}
by Proposition~\ref{prop:nulrang} and a $L_2$ Hodge splitting of $L_2(\R^n; \wedge^{k-1}\R^n\oplus \wedge^k\R^{n})$ adapted to $B_1$.
\end{proof}

Given Theorem~\ref{thm:main} and Proposition~\ref{prop:extsyst}, we can proceed as in \cite[Thm. 2.3]{AAM}, 
where the case $k=0$ was treated, to represent solutions to the exterior differential system  \eqref{eq:extssytem} with functional calculus as outlined in Section~\ref{sec:ops}.
To this end, define symbols
\begin{gather*}
  e^{-t\lambda}\chi^+(\lambda):=
  \begin{cases}
    e^{-t\lambda}, & \re\lambda>0, \\
    0, & \re \lambda\le 0,
  \end{cases}
  \qquad t\ge 0, \\
  e^{-t\lambda}\chi^-(\lambda):=
  \begin{cases}
    0, & \re\lambda\ge 0, \\
    e^{-t\lambda}, & \re \lambda< 0,
  \end{cases} 
  \qquad t\le 0.
\end{gather*}
For $t=0$, we obtain bounded spectral projections $\chi^\pm(T)$, with $\chi^+(T)+ \chi^-(T)$ being the projection 
onto $\clos{\ran(T)}$ along $\nul(T)$.
The following result roughly states that the spectral subspace $\chi^+(T)L_2:= \ran(\chi^+(T))$ is a Hardy type subspace 
containing traces of solutions to \eqref{eq:extssytem} in $\R^{1+n}_+$, whereas 
the spectral subspace $\chi^-(T)L_2:= \ran(\chi^-(T))$ is a Hardy type subspace 
containing traces of solutions to \eqref{eq:extssytem} in $\R^{1+n}_-$,
and the operators $e^{-tT}\chi^\pm(T)$ are Cauchy integral type operators, giving the value of the function at $(t,\cdot)$ from the boundary trace.

\begin{thm}
  Consider the exterior differential system \eqref{eq:extssytem}, with bounded, $t$-independent, accretive 
  coefficients $A_{k-1}, A_k, A_{k+1}$, and the associated operator $T=B_1D_1+D_2B_2$ as in Proposition~\ref{prop:extsyst}. 
  Given $\hf_0^\pm\in \chi^\pm(T)L_2$, the function $f\approx \hf$ defined by
$$
   \hf^\pm(t,x):= (e^{-tT}\chi^\pm(T)) \hf_0^\pm(x),\qquad (t,x)\in\R^{1+n}_\pm,
$$  
  is a solution to \eqref{eq:extssytem}, with limits $\lim_{t\to 0^\pm}\|\hf_t^\pm- \hf_0^\pm\|_2=0$
  and $\lim_{t\to \pm \infty}\|\hf_t^\pm\|_2=0$.
  Conversely, any solution $f^\pm\approx \hf^\pm$ to \eqref{eq:extssytem} in $\R^{1+n}_\pm$ with estimates $\sup_{t>0}\int_{t<\pm s<2t}\|f_s\|_2^2 ds<\infty$ is of this form, and in particular has the stated limits at $t=0$, for some $\hat f^\pm_0\in\chi^\pm(T)L_2$, and $t=\infty$.
  
  These solutions have square function, non-tangential maximal function and $L^t_\infty L^x_2$ estimates
$$
  \int_{\pm t>0}\|\pd_t \hf^\pm_t\|_2^2 tdt\approx \|\tN(\hf^\pm)\|_2^2\approx \sup_{\pm t>0} \|\hf^\pm_t\|_2^2\approx \|\hf^\pm_0\|_2^2.
$$
\end{thm}

The idea of proof is found in \cite[Thm. 3.2]{AAM} and \cite[Thm. 8.2]{AA1}.
In particular, the estimates follow from Theorem~\ref{thm:main}, using the symbol 
$\psi(\lambda)= \lambda e^{-\lambda}\chi^\pm(\lambda)$ for the square function estimates, and 
the symbol $\phi(\lambda)= e^{-\lambda}\chi^\pm(\lambda)$ for the non-tangential maximal function estimates and the $L^t_\infty L^x_2$ estimates. We omit the details.

\section{Square function estimates}   \label{sec:sqfcn}
In this section, we prove the square function estimates for the operator $T$ in Theorem~\ref{thm:main}.
We start by simplifying the problem with Lemma~\ref{lem:reductsqfcn}, and we use the following operators.
\begin{align*}
  P_t^1 &:= (I+t^2(B_1D_1)^2)^{-1}= \tfrac 12(R_t^1+ R_{-t}^1), \\
  Q_t^1 &:= tB_1D_1(I+t^2(B_1D_1)^2)^{-1}= \tfrac i{2}(R_t^1- R_{-t}^1), \\
  P_t^2 &:= (I+t^2(D_2B_2)^2)^{-1}= \tfrac 12(R_t^2+ R_{-t}^2), \\
  Q_t^2 &:= tD_2B_2(I+t^2(D_2B_2)^2)^{-1}= \tfrac i{2}(R_t^2- R_{-t}^2).
\end{align*}
It is known that these operators are uniformly bounded for $t>0$ and that square function estimates
$$
  \int_0^\infty \|Q_t^1 f\|_2^2 \frac {dt}t + \int_0^\infty \|Q_t^2 f\|_2^2 \frac {dt}t \lesssim \|f\|_2^2, \qquad
  f\in L_2(\R^n;\C^N),
$$
hold. See for example \cite[Thm. 3.4]{AAM}.

\begin{lem}   \label{lem:reductsqfcn}
Let
$$
   \Theta_t:= Q_t^2 R_t^1 B_1.
$$
Assume that we have square function estimates
$$
  \int_0^\infty \|\Theta_t f\|_2^2 \frac {dt}t \lesssim \|f\|_2^2,\qquad \text{for all } f\in \mH_1.
$$
Then we have square function estimates
$\int_0^\infty \|\psi(tT)f\|_2^2 \frac {dt}t \lesssim \|f\|_2^2$, $f\in L_2(\R^n;\C^N)$,
for $\psi$ as in Theorem~\ref{thm:main}.
\end{lem}

\begin{proof}
  It is known, see \cite{ADMc}, that it suffices to prove the square function estimate for $\psi(\lambda)= \lambda/(1+\lambda^2)$. For this $\psi$, we see from Proposition~\ref{prop:blockresolvent} that
$$
  \psi(tT)= \tfrac i2\big((I+itT)^{-1}-(I-itT)^{-1}\big)=
    \begin{bmatrix} 0 & 0 & 0 \\ 0 & Q^1_t & 0 \\ 0 & (R_{-t}^2-I)Q_t^1+ Q_t^2 R_t^1 & Q_t^2 \end{bmatrix}, \qquad t>0,
$$
by writing
\begin{multline*}
 \frac 12\left( \frac{tD_2B_2}{I+itD_2B_2}\frac{1}{I+itB_1D_1} + \frac{tD_2B_2}{I-itD_2B_2}\frac{1}{I-itB_1D_1} \right)
  \\ =
  \frac{itD_2B_2}{I-itD_2B_2}\frac{t B_1D_1}{I+t^2(B_1D_1)^2} + \frac{tD_2B_2}{I+t^2(D_2B_2)^2}\frac{1}{I+itB_1D_1}.
\end{multline*}
Since $R_{-t}^2$ are uniformly bounded and since we have square function estimates for $Q_t^1$ and $Q_t^2$, it suffices to prove square function estimates for $Q_t^2R_t^1$ on $B_1\mH_1$ as claimed.
\end{proof}

We now prove the square function estimates for $\Theta_t$ using techniques from the proof of the 
Kato square root estimate, following \cite[Sec. 4]{AAH}.

\begin{defn}
Let $\mD= \bigcup_{j\in \Z}\mD_{2^{-j}}$ denote the dyadic cubes in $\R^n$, with
$$
  \mD_t:= \sett{2^{-j}(0,1)^n+2^{-j}k}{k\in \Z^n},  \qquad 2^{-j-1}<t\le 2^{-j}, j\in\Z.
$$
Write $\ell(Q)$ for side length and $|Q|$ for measure of a cube $Q$.
Given $Q\in \mD$, write
$$
  A_0(Q):= Q, \qquad A_k(Q):= (2^k Q)\setminus (2^{k-1} Q), \quad k\ge 1,
$$
for the dyadic annuli around $Q$, where $aQ$ denote the cube with same center as $Q$ but with $\ell(aQ)= a\ell(Q)$.
\end{defn}

The key tool in the proof of the square function estimates, as well as for the non-tangential maximal function estimates, are the following $L_2$ off-diagonal estimates.

\begin{prop}   \label{prop:L2offdiag}
  For any $m<\infty$, there exists $C_m<\infty$ such that
$$
  \|\Theta_t f\|_{L_2(F)}\lesssim C_m (t/\dist(E,F))^m \|f\|_2,
$$
for all $f\in L_2(\R^n;\C^N)$ with $\supp f\subset E$, and any closed subsets
$E, F\subset \R^n$ such that $\dist(E,F):=\inf\sett{|x-y|}{x\in E, y\in F}>0$.
\end{prop}

\begin{proof}
  These estimates are known to hold for $R_t^j$,
  and therefore for $R_{-t}^j$, $P_t^j$ and $Q_t^j$, $j=1,2$,
   see for example \cite[Sec. 5]{elAAM}.
  From this, the estimates for $\Theta_t:= Q_t^2 (I- R_t^1) B_1$ follow as in \cite[Lem. 2.26]{AAH}.
\end{proof}

These $L_2$ off-diagonal estimates enable us to approximate the family of operators 
$\{\Theta_t\}_{t>0}$ by a family of multiplication operators $\{\gamma_t\}_{t>0}$,
where formally $\gamma_t= \Theta_t 1$. More precisely, we let
$$
  \gamma_t(x) v:= \sum_{k=0}^\infty \Theta_t (v \chi_{A_k(Q)})(x), \qquad x\in Q\in\mD_t, v\in \C^N,
$$
where $\chi_{A_k(Q)}$ denotes the characteristic function of the dyadic annulus $A_k(Q)$.
From Proposition~\ref{prop:L2offdiag}, we have the estimate
\begin{equation}   \label{eq:princpartbound}
  \|\gamma_t\|_{L_2(Q)} \lesssim \sum_{k=0}^\infty 2^{-km} 2^{kn/2} \le C
\end{equation}
uniformly for all $Q\in\mD_t$, $t>0$, by
choosing $m>n/2$.

We also need the following Sobolev--Poincar\'e inequality, see \cite[Sec. 7.8]{GT}.

\begin{lem}   \label{lem:poincare}
  Let $1< r<n$, or $r=n=1$, and $1/r^*= 1/r-1/n$.
  Assume that $1\le q\le r^*$ and $r\le p\le \infty$. 
   Then there exists $C<\infty$ such that for all  $u\in H^1_\loc(\R^n)$ and $0<r\le R<\infty$, we 
  have the estimate
$$
  \| u-u_{S}\|_{L_q(\Omega)}\le C |\Omega|^{1/q-1/p+1/n} R^n |S|^{-1} \|\nabla u\|_{L_p(\Omega)},
$$
for any convex set $\Omega$ with diameter $R$ and measure $|\Omega|$ and any measurable subset $S\subset \Omega$ with measure $|S|$.
\end{lem}

\begin{prop}    \label{prop:ppa}
  We have the estimate
$$
  \int_0^\infty \|\Theta_t f - \gamma_t E_t f\|_2^2 \frac {dt}t\lesssim \|f\|_2^2,\qquad\text{for all } f\in \mH_1,
$$
where $E_t$ denotes the dyadic averaging operator
$$
  E_t f(x)= E_Q f:= \frac 1{|Q|}\int_Q f(y) dy,\qquad x\in Q\in \mD_t.
$$
\end{prop}

\begin{proof}
  Let $P_t, Q_t$ denote the unpertubed operators $P_t^1, Q_t^1$, that is 
$$
  P_t:= (I+t^2 D_1^2)^{-1}, \qquad Q_t:= tD_1(I+t^2 D_1^2)^{-1}.
$$
Write
$$
  \Theta_t f - \gamma_t E_t f= \Theta_t(I-P_t)f+ (\Theta_t-\gamma_t E_t)P_tf+ (\gamma_t E_t)E_t(P_t-I)f=: I+II+III.
$$
For the first term we have
$$
  \|I\|_2= \|Q_t^2 (I-R_t^1) Q_t f\|_2\lesssim \|Q_t f\|_2,
$$
since $(I+itB_1D_1)^{-1}B_1(t^2D_1^2(I+t^2D_1^2))=(I+itB_1D_1)^{-1}(tB_1D_1)Q_t$,
and square function estimates for $Q_t$ give the desired estimate.

For the term $III$ we note from \eqref{eq:princpartbound}
that $\|\gamma_t E_t\|_{L_2\to L_2}\le C$.
Square function estimates for $E_t(P_t-I)$ can be proved as in \cite[Prop. 5.7]{AKMc}, replacing $\Pi$ there by the operator $D_1$.

The term $II$ we write
$$
  (\Theta_t-\gamma_t E_t)P_tf= \sum_{k\ge 0} \Theta_t\big((P_t f-E_Q(P_t f))\chi_{A_k(Q)} \big), \qquad\text{on } Q\in \mD_t.
$$
Proposition~\ref{prop:L2offdiag} and Poincar\'e's inequality in Lemma~\ref{lem:poincare} yields
\begin{multline*}
  \|II\|_2^2\lesssim \sum_{Q\in \mD_t}  \left(\sum_{k=0}^\infty 2^{-km} \|P_t f-E_Q(P_t f)\|_{L_2(A_k(Q))} \right)^2 \\
  \lesssim
  \sum_{Q\in \mD_t} \sum_{k=0}^\infty 2^{-km} \|P_t f-E_Q(P_t f)\|_{L_2(A_k(Q))}^2 \\
  \lesssim   \sum_{Q\in \mD_t} \sum_{k=0}^\infty 2^{-km+2k(n+1)} \|t\nabla P_t f\|_{L_2(A_k(Q))}^2
  \approx \|t\nabla P_t f\|_2^2,
\end{multline*}
if we choose $m$ sufficiently large.
Since $D_1$ is elliptic on $\ran(D_1)$, the square function estimate for $II$ follows from that for $Q_t$.
\end{proof}

To prove square function estimates for the remaining paraproduct term $\gamma_t E_t f$, we 
use the following test functions. 
For a small fixed parameter $\epsilon>0$, we define for all dyadic cubes $Q\in \mD$ and unit vectors $v\in\C^N$, the test function
$$
  f_Q^v:= (I+ (\epsilon\ell(Q)D_1B_1)^2)^{-1}(\eta_Q v),
$$
where $\eta_Q=1$ on $2Q$ and $\supp \eta_Q\subset (3Q)$, $\|\nabla \eta_Q\|_\infty\lesssim 1/\ell(Q)$.
The parameter $\epsilon$ is chosen so that the accretivity condition
$$
  \re\left( v, \int_Q f_Q^v dx \right) \ge |Q|/2 
$$
holds. This is possible since it is known that we have the estimate
$$
  \left| \frac 1{|Q|}\int_Q f_Q^v dx -v  \right|\lesssim \sqrt\epsilon.
$$
This can be proved by applying \cite[Lem. 5.6]{AKMc} with the operator $D_1$, similar to 
\cite[Lem. 5.10]{AKMc}.

\begin{prop}    \label{prop:Tb}  
  We have the estimate
$$
  \int_0^\infty\int_{\R^n}  |E_t f(x)|^2 |\gamma_t(x)|^2\frac {dxdt}t\lesssim \|f\|_2^2,\qquad\text{for all } f\in L_2(\R^n;\C^N).
$$
\end{prop}

\begin{proof}
  By Carleson's embedding theorem, it suffices to show that
$$
  \int_0^{\ell(Q)}\int_Q  |\gamma_t(x)|^2\frac {dxdt}t\lesssim |Q|,\qquad \text{for all } Q\in\mD.
$$
Following the proof of the Kato square root problem, see for example \cite[Sec. 5.3]{AKMc}, we now
do (1) a sufficiently fine sectorial decomposition of $\mL(\C^N)$, run (2) a stopping time argument
to construct an large sawtooth sub region of the Carleson box $Q\times (0,\ell(Q))$ where the test function
$f_Q^v$ is paraaccretive, and make (3) a John--Nirenberg bootstrapping argument for Carleson measure,
to show that it suffices to prove the estimate
$$
  \int_0^{\ell(Q)}\int_Q  |\gamma_t(x)E_t f_Q^v|^2\frac {dxdt}t\lesssim |Q|
$$
for all $Q\in\mD$ and unit vectors $v\in\C^N$.

We first note that 
\begin{multline*}
  \|\Theta_t f_Q^v\|_2= 
  \| Q_t^2 (B_1 f_Q^v)- it Q_t^2  R_t^1 B_1 (D_1B_1)(I+ (\epsilon\ell(Q)D_1B_1)^2)^{-1} (\eta_Q v)\|_2\\ \lesssim \| Q_t^2 (B_1 f_Q^v) \|_2+ t/\epsilon\ell(Q)\sqrt{|Q|},
\end{multline*}
since $Q_t^2$ and $R_t^1$ and $\epsilon\ell(Q) D_1B_1(I+ (\epsilon\ell(Q)D_1B_1)^2)^{-1} $ are uniformly bounded.
Therefore 
$$
    \int_0^{\ell(Q)}  \|\Theta_t f_Q^v\|^2\frac {dt}t\lesssim \|B_1f_Q^v\|_2^2 + |Q|\lesssim |Q|,
$$
so it suffices to prove
$$
  \int_0^{\ell(Q)}\int_Q  |(\Theta_t-\gamma_tE_t) f_Q^v|^2\frac {dxdt}t\lesssim |Q|.
$$
Although $f_Q^v\notin\mH_1$ in general, we have that
$f_Q^v-\eta_Q v= -\epsilon\ell(Q)D_1B_1 Q_{\epsilon\ell(Q)}^1 (\eta_Q v)\in \ran(D_1)$.
Thus Proposition~\ref{prop:ppa} applies to $f=f_Q^v-\eta_Q v$, and it remains to show
$$
  \int_0^{\ell(Q)}\int_Q  |(\Theta_t-\gamma_tE_t) (\eta_Q v)|^2\frac {dxdt}t\lesssim |Q|.
$$
We have on $Q$ that
$$
  (\Theta_t-\gamma_t E_t)(\eta_Q v)= \sum_{k\ge 0} \Theta_t\big(v(\eta_Q -1)\chi_{A_k(Q)} \big)
  = \sum_{k\ge 2} \Theta_t\big(v(\eta_Q -1)\chi_{A_k(Q)} \big),
$$
and Proposition~\ref{prop:L2offdiag} gives
$$
  \|(\Theta_t-\gamma_t E_t)(\eta_Q v)\|_{L_2(Q)}\lesssim \sum_{k=2}^\infty (t/2^k\ell(Q))^m (2^{nk}|Q|)^{1/2}\lesssim (t/\ell(Q))^{m} |Q|^{1/2}, 
$$
if $m$ is chosen large enough.
We obtain
$$
  \int_0^{\ell(Q)}\int_Q  |(\Theta_t-\gamma_tE_t) (\eta_Q v)|^2\frac {dxdt}t\lesssim  \int_0^{\ell(Q)} (t/\ell(Q))^{2m}|Q| \frac{dt}t\approx  |Q|
$$
and the proof is complete.
\end{proof}

\section{Non-tangential maximal function estimates}    \label{sec:maxfcn}

In this section, we prove the non-tangential maximal function estimates for the operator $T$ in Theorem~\ref{thm:main}.
We first consider the operator $D_2B_2$, that is the operator $T$ in the special case when $D_1=0$.
We prove the following result.

\begin{thm}    \label{thm:DBNT}
We have non-tangential maximal function estimates
$$
  \|\tN(\phi(tD_2B_2)f)\|_2 \lesssim \|f\|_2, \qquad f\in \mH_2,
$$
for $\phi$ as in Theorem~\ref{thm:main}.
\end{thm}

This result was proved in \cite[Prop. 2.56]{AAH} for operators $DB$ with $D$ of the form 
$D=\begin{bmatrix} 0 & \divv \\ -\nabla & 0 \end{bmatrix}$,
which appear in connection with boundary value problems for divergence form elliptic system. 
Below we give a simplified proof for general operators of the form $DB$.
Before doing so, we complete the proof of Theorem~\ref{thm:main} using Theorem~\ref{thm:DBNT}.

\begin{proof}[Proof of Theorem~\ref{thm:main}]
  Proposition~\ref{prop:blockresolvent} proves that $T$ is closed, densely defined and has the stated estimates of spectrum and resolvents.
  The square function estimates for $\psi(tT)$ follows from Propositions~\ref{prop:ppa} and \ref{prop:Tb}, using Lemma~\ref{lem:reductsqfcn}.
  Using Proposition~\ref{prop:duality}, the reverse square function estimates follow by duality.
  
  To prove the non-tangential maximal function estimates, we first note that estimates $\gtrsim$ hold since
  $\|\phi(tT)f-f\|_2\to 0$ as $t\to 0$ and 
$$
  \sup_{t>0} t^{-1}\int_t^{2t} \|\phi(sT)f\|_2^2 ds \lesssim \|\tN(\phi(tT)f)\|_2^2,
$$
as proved in \cite[Lem. 5.3]{AA1}.
 For the estimate $\lesssim$,
  it suffices to consider the resolvents $(I+itT)^{-1}$, since
$$
  \|\tN( \phi(tT)f-(I+itT)^{-1}f  )\|_2^2 \lesssim \int_0^\infty \| \phi(tT)f-(I+itT)^{-1}f  \|_2^2\frac{dt}t\lesssim \|f\|^2,
$$
where a proof of the first estimate can be found in \cite[Lem. 5.3]{AA1} and the second estimate follows from square function estimates with $\psi(\lambda):= \phi(\lambda)-(1+i\lambda)^{-1}$.
  By Proposition~\ref{prop:blockresolvent}, it remains to prove 
\begin{align}
  \|\tN( R_t^2 f)\|_2 &\lesssim \|f\|_2, \quad f\in \mH_2, \label{eq:est1} \\
    \|\tN( R_t^1 f)\|_2 &\lesssim \|f\|_2, \quad f\in B_1\mH_1, \label{eq:est2}  \\
      \|\tN( R_t^2 R_t^1 f )\|_2 &\lesssim \|f\|_2, \quad f\in B_1\mH_1. \label{eq:est3} 
\end{align}
  Estimate \eqref{eq:est1} follows from Theorem~\ref{thm:DBNT}.
  For \eqref{eq:est2}, we  estimate
$$
  \|\tN( R_t^1 B_1 v)\|_2= \|\tN(  B_1 (I+itD_1B_1)^{-1} v)  \|_2\lesssim \|v\|_2\lesssim \|B_1 v\|_2,\qquad v\in\mH_1,
$$
using that $B_1$ is a bounded multiplication operator and  Theorem~\ref{thm:DBNT}, with $D_2B_2$ replaced by $D_1B_1$, for the first estimate, and that $B_1$ is accretive on $\mH_1$ for the second estimate.
  
  To prove \eqref{eq:est3}, we introduce the auxiliary non-tangential maximal functions
$$
  \tN^k u(x):= \sup_{t>0} \left( 2^{-kn} t^{-(1+n)} \iint_{B(x,2^k t)\times (t/2, t)} |u(s,y)|^2 dsdy \right)^{1/2},
  \qquad k\ge 0.
$$
From $L_2$ off-diagonal estimates for $R_t^2$, as in Proposition~\ref{prop:L2offdiag}, we get that
$$
  \tN^0(R_t^2 R_t^1 f)(x)\lesssim \sum_{k=0}^\infty 2^{-km} \tN^k( R_t^1 f)(x).
$$
We claim that 
$$
  \| \tN^k( R_t^1 f) \|_2 \lesssim  2^{kn/2} \| \tN^0(R_t^1 f) \|_2.
$$
Thus, choosing $m$ large, we obtain the desired estimates
$$
  \| \tN^0(R_t^2 R_t^1 f) \|_2 \lesssim \| \tN^0(R_t^1 f) \|_2 \lesssim \|f\|_2.
$$
To prove the claim, assume that $\tN^k g(x)>\lambda$.
It follows that there exists $x_1\in B(x_0, 2^k t)$ such that $t^{-(1+n)}\iint_{B(x_1, t/2)\times (t/2,t)}|g|^2\gtrsim \lambda^2$.
We conclude that for all $y\in B(x_1, t/2)$ we have $\tN^0 g(y)\gtrsim \lambda$,
and therefore
$$
  |\sett{x}{\tN^k g(x)>\lambda}|\le |\sett{x}{M(\chi_{\sett{y}{\tN^0 g(y)\gtrsim \lambda}})(x)\gtrsim 2^{-kn}}|\lesssim 2^{kn}|\sett{y}{\tN^0 g(y)\gtrsim\lambda}|,
$$
using the weak $L_1$ boundedness of the Hardy--Littlewood maximal function $M$, from which the claim follows,.
\end{proof}

We now turn to the proof of Theorem~\ref{thm:DBNT}, where we write $D:= D_2$ and $B:= B_2$ to simplify notation.
We use the following version of the Caccioppoli estimate for equations of the form $\pd_t v+ BDv=0$.

\begin{lem}  \label{lem:cacc}
  Let $\eta\in C_0^\infty(\R^{1+n})$ and consider a function $v$ solving $\pd_t v+ BD v=0$ on a neighbourhood of $\supp\eta$.
  Then 
  $$
    \iint |Dv(t,x)|^2|\eta(t,x)|^2 dtdx\lesssim \iint |v(t,x)|^2 |\nabla \eta(t,x)|^2 dtdx.
  $$
\end{lem}

\begin{proof}
  Integrating by part twice, using anti-symmetry of $\pd_t$ and symmetry of $D$, we obtain
\begin{multline*}
  \iint (BDv, Dv)\eta^2 = -\iint (\pd_t v, Dv)\eta^2 \\
  = \iint (v, \pd_t D v)\eta^2+ 2\iint (v,Dv)\eta\pd_t \eta 
   =-\iint (v, DBD v)\eta^2+ 2\iint (v,Dv)\eta\pd_t \eta \\
   =  -\iint (D(\eta^2v), BD v)+ 2\iint (v,Dv)\eta\pd_t \eta.
\end{multline*}
We get the estimate
$$
 \re\iint (BDv, Dv)\eta^2\lesssim \iint |v||Dv| |\eta||\nabla\eta|.
$$
For each $t\in\R$, we have by the accretivity of $B$ that
$$
  \int |D(\eta v)|^2 dx \lesssim \re \int (BD(\eta v), D(\eta v)) dx.
$$
Integrating both sides with respect to $t$ and using the product rule for the derivatives, we get
$$
   \iint |Dv|^2\eta^2 dtdx\lesssim \iint (|Dv|\eta)( |v||\nabla\eta| ) dtdx+ \iint |v|^2 |\nabla \eta|^2 dtdx.
$$
Using the absorption inequality, we obtain the stated estimate.
\end{proof}

\begin{proof}[Proof of Theorem~\ref{thm:DBNT}]
 To prove the non-tangential maximal function estimates, it suffices to estimate the semigroups $e^{-tDB}\chi^+(DB)$ and $e^{tDB}\chi^-(DB)$, since
\begin{multline*}
  \|\tN( \phi(tDB)f-e^{-tDB}\chi^+(DB)f-e^{tDB}\chi^-(DB)f  )\|_2^2 \\
  \lesssim \int_0^\infty \| \phi(tT)f-e^{-tDB}\chi^+(DB)f-e^{tDB}\chi^-(DB)f \|_2^2\frac{dt}t\lesssim \|f\|^2,
\end{multline*}
where the second estimate follows from square function estimates for $DB$ with $\psi(\lambda):= \phi(\lambda)-e^{-|\lambda|}$.
Moreover, by a limiting argument, we may assume that $f\in\ran(D)$.

Consider first $e^{-tDB}\chi^+(DB)f$ and a Whitney region $W= B(x_0, t_0)\times (t_0/2, t_0)$.
Write $\chi^+(DB)f=Dv$ with $v\in \chi^+(BD)L_2\cap\dom(D)$.
Let $P$ be the orthogonal projection onto $\clos{\ran(D)}$, and denote by
$[Pv]:= |B(x_0,t_0)|^{-1}\int_{B(x_0,t_0)} Pv(x) dx$ the average of $Pv\in \clos{\ran(D)}$.
Note that $Dv= D(Pv)$ since $P$ projects along the null space $\nul(D)$.
Define the function
$$
  w(t,x):= t\psi_1(tBD) BD v(x)+ \phi_1(tBD)(Pv-[Pv])(x), \qquad (t,x)\in\R^{1+n}_+,
$$
where
\begin{align*}
  \psi_1(\lambda)&:= \lambda^{-1}\big(e^{-\lambda}-(1-\lambda)/(1+\lambda^2)\big) \chi^+(\lambda), \\
  \phi_1(\lambda)&:= (1-\lambda)/(1+\lambda^2).
\end{align*}
In the second term, $[Pv]$ is regarded as a constant function. 
Since $\phi_1(tBD)$ is a linear combination of resolvents, it extends to a bounded operator $L_\infty(\R^n)\to L_2(B(x_0, t_0))$ as in the estimate \eqref{eq:princpartbound}.
For the first term, the decay of $\psi_1$ at $\lambda=0$ and $\infty$ shows that we have square function estimates for $\psi_1(tBD)$.
We calculate for $t>0$ that
\begin{align*}
  Dw &= D e^{-tBD} v= e^{-tDB}\chi^+(DB) f, \\
  \pd_t w &=-BDe^{-tBD}v,
\end{align*}
using $D\phi_1(tBD)[Pv]= \pd_t \phi_1(tBD)[Pv]=0$.
Since $\pd_t w+BD w=0$, we get from Lemma~\ref{lem:cacc} that
$$
  \iint_W |e^{-tDB}\chi^+(DB)f|^2 dtdx=   \iint_W |Dw|^2 dtdx \lesssim \iint_{\widetilde W}|t^{-1} w|^2 dtdx,
$$
with a slightly enlarged Whitney region $\widetilde W= B(x_0, 2t_0)\times (t_0/4, 2t_0)$.
Using square function estimates, we obtain
$$
  \|\tN(\psi_1(tBD)BDv)\|_2^2\lesssim \int_0^\infty \| \psi_1(tBD)BDv \|_2^2 \frac {dt}t\lesssim \|BDv\|_2^2\lesssim \|f\|_2^2,
$$
so it remains to estimate $\phi_1(tBD)(Pv-[Pv])$.
To this end, for fixed $t_0/4<t<2t_0$, use the $L_2$ off-diagonal estimates for the resolvents of $BD$ to get
\begin{multline*}
  \| \phi_1(tBD)(Pv-[Pv]) \|^2_{L_2(B(x_0, 2t_0))} \lesssim \sum_{k=0}^\infty 2^{-km}\| Pv-[Pv]\|^2_{L_2(B(x_0,2^k 2t_0))}
  \\
  \lesssim \sum_{k=0}^\infty 2^{k(-m+3n-2n/p+2)}t_0^{n-2n/p+2}  \| \nabla v\|^2_{L_p(B(x_0,2^k 2t_0))},
\end{multline*}
using the Sobolev--Poincar\'e inequality from Lemma~\ref{lem:poincare} with suitable $1\le p<2=q$, $\Omega= B(x_0,2^k 2t_0)$
and $S= B(x_0, 2t_0)$.
Integrating with respect to $t$, we get
\begin{multline*}
  t_0^{-1-n}\iint_{\widetilde W}|t^{-1} \phi_1(tBD)(Pv-[Pv])|^2 dtdx\\
   \lesssim \sum_{k=0}^\infty 2^{k(-m+3n+2)} M(|\nabla Pv|^p)(x_0)^{2/p}\lesssim  M(|\nabla Pv|^p)(x_0)^{2/p},
\end{multline*}
choosing the parameter $m$ large.
Therefore, boundedness of the Hardy--Littlewood maximal function $M$ on $L_{2/p}$ gives
$$
  \|\tN(e^{-tDB}\chi^+(DB)f)\|_2^2\lesssim\|f\|_2^2+ \|\nabla Pv\|_2^2\lesssim \|f\|_2^2+ \|DPv\|_2^2\lesssim \|f\|_2^2,
$$
by the ellipticity of $D$ on $\clos{\ran(D)}$.

The estimate of the semigroup $e^{tDB}\chi^-(DB)$ follows from the above estimate upon replacing $D$ by $-D$, since $e^{-t(-DB)}\chi^+(-DB)= e^{tDB}\chi^-(DB)$.
This proves the non-tangential maximal function estimate for $\phi(tDB)$.
\end{proof}

\bibliographystyle{acm}

\end{document}